\documentclass[11pt]{amsart}
\usepackage{amssymb}
\usepackage{amscd}
\usepackage{color}
\usepackage{url}
\usepackage{hyperref}
\usepackage{xcolor}
\usepackage{graphicx}

\numberwithin{equation}{section}

\DeclareMathOperator*{\Span}{span \;}

\def\re{\mathop{\rm Re}\nolimits}
\def\ker{\mathop{\rm ker}\nolimits}

\theoremstyle{plain}
\newtheorem{theorem}{Theorem}[section]
\newtheorem{thm}[theorem]{Theorem}
\newtheorem{thm*}{Theorem}

\newtheorem{cor}[theorem]{Corollary}

\newtheorem{prop}[theorem]{Proposition}
\theoremstyle{definition}

\newtheorem{ex}[theorem]{Example}
\theoremstyle{remark}
\newtheorem{rem}[theorem]{Remark}

\def\CC{\mathbb C}
\def\DD{\mathbb D}
\def\NN{\mathbb N}
\def\RR{\mathbb R}
\def\TT{\mathbb T}

\def\LL{{\mathcal L}}
\def\FF{\mathcal F}

\def\beginpf{\begin{proof}}
\def\endpf{\end{proof}}
\def\beq{\begin{equation}}
\def\eeq{\end{equation}}

\newcommand{\xdownarrow}[1]{%
  {\left\downarrow\vbox to #1{}\right.\kern-\nulldelimiterspace}
}

\begin{document}

\title[Ces\`aro operator]{Insights on the Ces\`aro operator: shift semigroups \newline and invariant subspaces}

\author{Eva A. Gallardo-Guti\'{e}rrez}
\address{Eva A. Gallardo-Guti\'errez \newline
Departamento de An\'alisis Matem\'atico y Matem\'atica Aplicada,\newline
Facultad de Matem\'aticas,
\newline Universidad Complutense de
Madrid, \newline
Plaza de Ciencias 3, 28040 Madrid,  Spain
\newline
and Instituto de Ciencias Matem\'aticas ICMAT,
\newline Madrid,  Spain }
\email{eva.gallardo@mat.ucm.es}

\author{Jonathan R. Partington}
\address{Jonathan R. Partington, \newline
School of Mathematics, \newline
University of Leeds, \newline
Leeds LS2 9JT, United Kingdom}
\email{J.R.Partington@leeds.ac.uk}

\thanks{Both authors are partially supported by Plan Nacional  I+D grant no. PID2019-105979GB-I00, Spain. The first author is also supported by
 the Spanish Ministry of Science and Innovation, through the ``Severo Ochoa Programme for Centres of Excellence in R\&D'' (CEX2019-000904-S) and from the Spanish National Research Council, through the ``Ayuda extraordinaria a Centros de Excelencia Severo Ochoa'' (20205CEX001). }


\date{June 2022,  revised version September 24th, 2022}



\begin{abstract}
A closed subspace is invariant under the Ces\`aro operator $\mathcal{C}$ on the classical Hardy space $H^2(\DD)$ if and only if its orthogonal complement is invariant under the $C_0$-semigroup of composition operators induced by the affine maps $\varphi_t(z)= e^{-t}z + 1 - e^{-t}$ for $t\geq 0$ and $z\in \DD$. The corresponding result also holds in the Hardy spaces $H^p(\DD)$ for $1<p<\infty$. Moreover, in the Hilbert space setting, by linking the invariant subspaces of $\mathcal{C}$ to the lattice of the closed invariant subspaces of the standard right-shift semigroup acting on a particular weighted $L^2$-space on the line, we exhibit a large class of non-trivial closed invariant subspaces and provide a complete characterization of the finite codimensional ones, establishing, in particular, the limits of such an approach towards describing the lattice of all invariant subspaces of $\mathcal{C}$. Finally, we present a functional calculus argument which allows us to extend a recent result by Mashreghi, Ptak and Ross regarding the square root of $\mathcal{C}$ and discuss its invariant subspaces.
\end{abstract}


\maketitle

\section{Introduction and preliminaries}
Despite the fact that one of the most classical transformations of sequences is the Ces\`aro operator $\mathcal{C}$, there are still many questions about it unsettled.  Recall that $\mathcal{C}$ takes a complex sequence $\textbf{a}=(a_0, a_1, a_2\dots )$ to that with $n$-th entry:
$$
(\mathcal{C}\, \textbf{a})_n= \frac{1}{n+1} \sum_{k=0}^n a_k, \qquad (n\geq 0).
$$
Upon identifying sequences with Taylor coefﬁcients of power series, $\mathcal{C}$ acts \emph{formally} on  $f(z)=\sum_{k=0}^{\infty} a_k z^k$ as
\begin{equation}\label{definition Cesaro}
\mathcal{C}(f)(z)=\sum_{n=0}^\infty \left (\frac{1}{n+1} \sum_{k=0}^n a_k \right ) z^n.
\end{equation}
Indeed, if $f$ is a holomorphic function on the unit disc $\mathbb{D}$ so is  $\mathcal{C}(f)$ and moreover, $\mathcal{C}$ is an isomorphism of the Fr\'echet space $\mathcal{H}(\DD)$ of all holomorphic functions on $\DD$ endowed with the topology of uniform convergence on compacta.

\smallskip

Nevertheless, this is no longer true when $\mathcal{C}$ is restricted to the classical Hardy spaces $H^p(\DD)$, $1\leq p <\infty$. A classical result of Hardy concerning trigonometric series along with M. Riesz’s theorem yields that $\mathcal{C}$ is bounded on $H^p(\DD)$ for $1<p <\infty$. Likewise, Siskakis proved that $\mathcal{C}$ is bounded on $H^1(\DD)$ (providing even an alternative proof of the boundedness on $H^p(\DD)$ for $1<p <\infty$; see \cite{siskakis}, \cite{siskakis2}). However 0 belongs to the spectrum of $\mathcal{C}$ in $H^p(\DD)$ and hence, $\mathcal{C}$ is not an isomorphism \cite{siskakis}.

\smallskip

Note that \eqref{definition Cesaro} can be written as
$$
\mathcal{C}(f)(z)= \left \{ \begin{array}{ll}
\displaystyle \frac{1}{z} \int_0^z \frac{f(\xi)}{1-\xi} \, d\xi, & z\in \DD\setminus\{ 0\}, \\
\noalign{\medskip}
f(0)&  z=0;
\end{array} \right.
$$
for $z\in \DD$. There is an extensive literature on the Ces\`aro operator, and more general, on integral operators, acting on a large variety of spaces of analytic functions regarding its boundedness, compactness or spectral picture (see \cite{Ale06} or \cite{AlSi}, for instance).

\smallskip

If we restrict ourselves to the Hilbert space case $H^2(\DD)$, Kriete and Trutt proved the striking result that the Ces\`aro operator is \emph{subnormal}, namely, $\mathcal{C}$ on $H^2(\DD)$ has a normal extension. More precisely, if $I$ denotes the identity operator on $H^2(\DD)$, they proved that $I-\mathcal{C}$ is unitarily equivalent to the operator of multiplication by the identity function acting on the closure of analytic polynomials on the space $L^2(\mu, \DD)$ for a particular measure $\mu$ (see \cite{KT71}). An alternative proof of the Kriete and Trutt theorem, based on the connection between $\mathcal{C}$ and composition operators semigroups, was later established by Cowen \cite{Co}.

\smallskip

For $H^p(\DD)$, $1<p<\infty$, Miller, Miller and Smith \cite{MMS} showed that $\mathcal{C}$ is \emph{subdecomposable}, namely, it has a decomposable extension (the  $H^1(\DD)$ case was proved by Persson \cite{Per} ten years later). Decomposable operators were introduced by Foia\c{s} \cite{FOIAS} in the sixties as a generalization of spectral operators in the sense of Dunford, and many spectral operators in Hilbert spaces as unitary operators, self-adjoint operators or more generally, normal operators are decomposable (see the monograph \cite{LN00} for more on the subject).

\smallskip

Normal operators on Hilbert spaces or more generally, decomposable operators on Banach spaces have a rich lattice of non-trivial closed invariant subspaces with a significant description of them. But, very little is known about this description even for concrete examples of subnormal operators as the Ces\`aro operator, and this will be the main motivation of the present manuscript.

\smallskip

In this context, we discuss invariant subspaces of the Ces\`aro operator $\mathcal{C}$ on the Hardy space $H^2(\DD)$. Broadly speaking, we prove a   \emph{Beurling--Lax Theorem} for the Ces\`aro operator and provide a complete characterization of the finite codimensional invariant subspaces of $\mathcal{C}$.
The composition semigroup method has turned out to be a powerful tool to study the Ces\`aro operator and we will make use of such technique in Section \ref{section 2} to link the invariant subspaces of $\mathcal{C}$ to those of the right-shift semigroup $\{S_{\tau}\}_{\tau\geq 0}$ acting on a particular weighted $L^2(\mathbb{R}, w(y)dy)$.  In particular, we will establish the limits of our approach towards describing completely the lattice of the invariant subspaces of $\mathcal{C}$.

\smallskip

In Section \ref{section 3}, we discuss Phillips functional calculus (as in Haase’s book \cite{haase}) which will allow us, in particular, generalize the recent work  by Mashreghi, Ptak and Ross \cite{MPR} regarding the square roots of $\mathcal{C}$. In particular, we will discuss their invariant subspaces.

\smallskip

In order to close this introductory section we collect some preliminaries for the sake of completeness.

\subsection{Semigroups of composition operators}
The study of semigroups of composition operators on various function spaces of analytic functions has its origins in the work of Berkson and Porta \cite{Ber-Por}, where they characterize their generators on $H^p(\DD)$, proving, indeed, that these semigroups are always strongly continuous.

Recall that a one-parameter family $\Phi=\{\varphi_t\}_{t\ge 0}$ of analytic self-maps of $\mathbb{D}$ is called
a \textit{holomorphic flow} (or \textit{holomorphic semiflow} by some authors) if it is a continuous
family that has a semigroup property with respect to composition, namely
\begin{enumerate}
\item[1)]
$\varphi_0(z)=z$, for $z\in \DD$;

\item[2)] $\varphi_{t+s}(z)=\varphi_{t}\circ \varphi_{s}(z)$, for $t,s\ge0$, and $z\in \DD$;

\item[3)] For any $s\ge0$ and any $z\in\DD$,
$\lim_{t\to s}\varphi_{t}(z)=\varphi_{s}(z)$.
\end{enumerate}
The holomorphic flow  $\Phi$ is trivial if $\varphi_t(z) =z$ for all $t\geq 0$. Otherwise, we say that $\Phi$ is nontrivial. We refer to the recent monograph \cite{BCD} for more on the subject.

\smallskip

Associated to the holomorphic flow $\Phi=\{\varphi_t\}_{t\ge 0}$ is the family of composition operators $\{C_{\varphi_t} \}_{t\geq 0}$, defined
on the space of analytic functions on $\mathbb{D}$ by
$$
C_{\varphi_t} f= f\circ \varphi_t.
$$
Clearly, $\{C_{\varphi_t} \}_{t\geq 0}$ has the semigroup property:
\begin{enumerate}
\item[1)] $C_{\varphi_0}=I$;

\item[2)]$C_{\varphi_t} C_{\varphi_s}=C_{\varphi_{t+s}}$ for all $t,s\ge 0$.
\end{enumerate}
Moreover, recall that if an operator  semigroup $\{T_t\}_{t\geq 0}$ acts on a Banach space
$X $, then it is called \emph{strongly continuous} or \emph{$C_0$-semigroup}, if it satisfies
$$
\lim_{t\to 0^+} T_t f=f
$$
for any $f\in X $. Given a $C_0$-semigroup $\{T_t\}_{t\geq0}$ on a Banach space $X $, recall that its generator is the  closed and
densely defined linear operator $A$ defined by
\[
Af=\lim_{t\to 0^+} \frac {T_t f-f} t
\]
with domain $\mathcal{D}(A)=\{f\in X : \lim_{t\to 0^+} \frac {T_t f-f} t \enspace \text{exists}\}$.

\smallskip

\section{The lattice of the invariant subspaces of the Ces\`aro operator}\label{section 2}

The aim of this section is identifying the lattice of the invariant subspaces of the Ces\`aro operator $\mathcal{C}$ acting on the Hardy space $H^2(\DD)$. In particular, we will characterize the finite codimensional invariant subspaces of $\mathcal{C}$.

\smallskip

Our first result resembles a \emph{Beurling--Lax Theorem} for the Ces\`aro operator.

\smallskip

\begin{thm}\label{thm semigroup}
Let $\Phi=\{\varphi_t\}_{t\ge 0}$ be the holomorphic flow given by
\begin{equation}\label{semigroup C}
\varphi_t(z)= e^{-t}z + 1 - e^{-t}, \qquad (z\in \DD).
\end{equation}
A closed subspace $M$ in $H^2(\DD)$ is invariant under the Ces\`aro operator if and only if its orthogonal complement $M^{\perp}$ is invariant under the semigroup of composition operators induced by $\Phi$, namely, $\{C_{\varphi_t} \}_{t\geq 0}$.
\end{thm}

Before proceeding with the proof, note that each $\varphi_t$ in \eqref{semigroup C} is a affine map which is a hyperbolic non-automorphism of the unit disc
inducing  a bounded composition operator $C_{\varphi_t}$ on $H^2(\DD)$ with norm
\begin{equation}\label{norm estimate}
\|C_{\varphi_t}\|_2= e^{\frac{t}{2}},
\end{equation}
(see, for instance, \cite[Theorem 9.4]{CMc}).

\noindent Likewise,  the generator of the $C_0$-semigroup $\{C_{\varphi_t} \}_{t\geq 0}$ is given by
\[
Af(z)=(1-z)f'(z), \qquad (z\in \DD),
\]
(see the pioneering work by Berkson and Porta \cite{Ber-Por}, for instance).
\medskip

\begin{proof}
First, let us show that the cogenerator of the $C_0$-semigroup $\{C_{\varphi_t} \}_{t\geq 0}$ given by
$$V=(A+I)(A-I)^{-1}$$
is a well-defined bounded operator. For such a task, we will prove that $1 \in \rho(A)$, the resolvent of $A$, or equivalently,
$$A-I: \mathcal{D}(A) \subset H^2(\DD) \to H^2(\DD)$$
is bijective.

\smallskip

Note that $A-I$ is an injective operator in $\mathcal{D}(A) \subset H^2(\DD)$. Indeed, if $(A-I)f(z)=0$ then $(1-z)f'(z)-f(z)=0$, so $f(z)=C/(1-z)$ for some complex constant $C \in \CC$. But for $C \ne 0$ we have that $f \not\in H^2(\DD)$.

\smallskip

We claim that $A-I$ is also surjective. Given $g \in H^2(\DD)$,  in order to find $f \in H^2(\DD)$ such that
\[
(A-I)f(z)=g(z) \qquad (z \in \DD),
\]
or,
\[
(1-z)f'(z)-f(z)=((1-z)f(z))'= g(z) \qquad (z \in \DD),
\]
let
\begin{equation}\label{def f}
f(z)=\frac{1}{1-z}\int_1^z g(u) \, du = \frac{1}{z-1} \int_z^1 g(u) \, du
\end{equation}
for $z \in \DD$.

\smallskip

Note that the adjoint of the Ces\`aro operator $\mathcal{C}^*$ has the following matrix with respect to the canonical
orthonormal basis of $H^2(\DD)$:
\[
\begin{pmatrix}
1 & \frac12 & \frac13  & \frac14 & \ldots\\
0 & \frac12 & \frac13 & \frac14 & \ldots\\
0 &  0 & \frac13 & \frac14 & \ldots\\
0 & 0 & 0 & \frac14 & \ldots \\
\vdots & \vdots & \vdots & \vdots & \ddots
\end{pmatrix}
\]

Writing
\[
Tg(z)= \frac{1}{z-1} \int_z^1 g(u) \, du,
\]
for $g\in H^2(\DD)$, observe that
\[
T z^n = \frac{1}{z-1} \left( \frac{1-z^{n+1}}{n+1} \right) = - \frac{1+z+\ldots+z^n}{n+1}.
\]
Accordingly, $T$ is a well-defined operator in $H^2(\DD)$ (as it is -$\mathcal{C}^*$). This in particular implies that the function $f$ in \eqref{def f} belongs to $H^2(\DD)$ and  hence $A-I$ is surjective.

\smallskip

Accordingly, the cogenerator $V$ of the $C_0$-semigroup $\{C_{\varphi_t} \}_{t\geq 0}$ is a well-defined bounded operator on $H^2(\DD)$.

\smallskip

Now, having in mind the norm estimate \eqref{norm estimate}, we observe that the $C_0$-semigroup
$\{e^{-t} C_{\varphi_{2t}} \}_{t\geq 0}$ is contractive on $H^2(\DD)$ and its generator is $2A-I$. Since
$$V-I = 2(A-I)^{-1} = {4} ((2A-I)-I)^{-1} = -2\mathcal{C}^*$$
the invariant subspaces of the cogenerator are simply the common invariant
subspaces of the semigroup (see \cite[Chap. 10, Theorem 10.9]{fuhrmann}) and
the statement of the theorem  follows.
\end{proof}

\smallskip

First, let us remark that a similar argument in the context of $C_0$-semigroups of analytic 2-isometries was also used in \cite{GP18}. Likewise, recalling that the Hardy space $H^p(\DD)$, $1\leq p<\infty$,
consists of holomorphic functions $f$ on $\mathbb D$ for which the
norm
$$
\|f\|_p=\left ( \sup_{0\leq r<1} \int_{0}^{2\pi}
|f(re^{i\theta})|^p \, \frac{d\theta}{2\pi}\right )^{1/p}
$$
is finite, we note that the previous proof also works in $H^p(\DD)$-spaces ($1<p<\infty$) with the natural identification of the dual space $H^p(\DD)^* \cong H^{p'}(\DD)$ where $p'$ is the conjugate exponent: $\frac{1}{p}+\frac{1}{p'}=1$.
In this case, the bounded composition operator $C_{\varphi_t}$ on $H^p(\DD)$ has norm
\begin{equation}\label{norm estimate 2}
\|C_{\varphi_t}\|_p= e^{\frac{t}{p}}
\end{equation}
for $1\leq p<\infty$ (see \cite[Exercise 3.12.5, p. 56-57]{nik1}
and \cite[Part Two, Ch. 1, Section 5 1. B p.~165-166]{HaJo}, for instance). Accordingly, the $C_0$-semigroup {$\{e^{-(p'-1)t} C_{\varphi_{p't}} \}_{t\geq 0}$ is also contractive on $H^p(\DD)$ with generator $p' A-(p'-1)I$ and cogenerator $I+2(p'A-(p'-1)I-I)^{-1}=I+2(p'(A-I))^{-1}$}. Therefore, a closed subspace $M$ in $H^p(\DD)$  for $1<p<\infty$ is invariant under the Ces\`aro operator if and only if its annihilator $M^{\perp}$ in $H^{p'}(\DD)$ is invariant under the  $C_0$-semigroup $\{C_{\varphi_t} \}_{t\geq 0}$.

\smallskip

\noindent In this regard, it is worth noting that $(1-z)^{-1} \not \in H^p(\DD)$ for any $1\leq p\leq \infty$ since $(1-e^{i\theta})^{-1} \not \in L^p(\TT)$.

\begin{rem}
By Dunford and Schwartz \cite[Thm.~11, p.~622]{DS}, the resolvent can be expressed in terms of the Laplace transform of the semigroup; that is,
\[
(A-I)^{-1} f(z)= \int_0^\infty e^{-t} C_{\varphi_t}f(z) \, dt, \qquad (z \in \DD).
\]
So
\[
\mathcal{C}^* f(z)= - \int_0^\infty e^{-t} f(e^{-t}z+1-e^{-t}) \, dt, , \qquad (z \in \DD).
\]
\end{rem}

\smallskip

Indeed, a consequence of the previous formula is the following:

\begin{cor}\label{coro adjunto}
A closed subspace $M$ in $H^2(\DD)$ is invariant under the Ces\`aro operator if and only if it is invariant under the semigroup  $\{C_{\varphi_t}^* \}_{t\geq 0}$.
\end{cor}

Alternatively, one may use that the adjoint of the generator is the generator of the
adjoint semigroup in the context of Hilbert spaces (see \cite[Chap. 10]{fuhrmann}, for instance).

\medskip

Finally, note that the adjoint $C_{\varphi_t}^*$ in $H^2(\DD)$ in Corollary \ref{coro adjunto} may be explicitly computed as a weighted composition operator (see \cite[Theorem 9.2]{CMc}). Indeed, expressing $\varphi_t(z)= e^{-t}z + 1 - e^{-t}$ in its normal form, namely $\varphi_t(z)= (a_t z +b_t)/(0 z+ a_t^{-1})$ where
$$a_t=e^{-t/2},\qquad b_t=\frac{1-e^{-t}}{e^{-t/2}}$$
we deduce that $C_{\varphi_t}^*= T_{g_t} C_{\sigma_t} T_{h_t}^*$ where
$$
g_t(z)= \frac{1}{-b_t z+ a_t^{-1}}=\frac{e^{-t/2}}{1-(1-e^{-t})z}, \qquad (z\in \DD),
$$
$$
\sigma_t(z)= \frac{a_tz}{-b_t z+ a_t^{-1}}=\frac{e^{-t}z}{1-(1-e^{-t})z}, \qquad (z\in \DD),
$$
and
$$
h_t(z)= a_t^{-1}= e^{t/2};
$$
and $T_{g_t}$ and $T_{h_t}$ denotes the analytic Toeplitz operators acting on $H^2(\DD)$ induced by the symbols $g_t$ and $h_t$ respectively.

Accordingly,
$$C_{\varphi_t}^* f(z)= \frac{1}{1-(1-e^{-t})z} f\left (\frac{e^{-t}z}{1-(1-e^{-t})z} \right ), \qquad (z\in \DD)$$
for $f\in H^2(\DD)$ and every $t\geq0$.

\medskip

\subsection{Shift semigroups}

In order to provide a characterization of the finite codimensional invariant subspaces of $\mathcal{C}$, we will make use of a semigroup of operators acting on the
Hardy space of the right half-plane $\CC_+$. Recall that the Hardy space $H^2(\CC_+)$ consists of the functions $F$ analytic
on $\CC_+$ with finite norm
$$
\Vert F \Vert_{H^2 (\CC_+)}  = \left\{\sup_{0<x<\infty}
\int_{-\infty}^{\infty} \vert F (x+iy) \vert ^2\, dy\right\}^{1/2}.
$$
The classical Paley--Wiener Theorem (see \cite{Ru}, for instance) states that $H^2(\CC_+)$ is isomorphic under the Laplace transform  to $L^2(\RR_+)$, the space of measurable functions square-integrable over $(0,\infty)$. More precisely,  to each  function $F \in H^2(\CC_+)$ there
corresponds a function  $f \in L^2(\mathbb{R}_+)$ such that
$$
F(s)= (\LL f)(s):= \int_0^\infty f(x) e^{-sx}\, dx\,, \qquad (s \in \CC_+),
$$
and
$$
\Vert F\Vert^2 _{H^2 (\CC_+)}= 2\pi \int_0^\infty \vert f(x) \vert ^2 \,dx\,.
$$

\smallskip

A first observation already stated in \cite[Lemma 4.2]{CN} is that each
$\varphi_t(z)= e^{-t}z + 1 - e^{-t}$  for $z\in \DD$ and $t>0$  induces a composition operator in $H^2(\DD)$ which is similar under an isomorphism $\mathcal{U}$ (indeed unitarily equivalent up to a constant) to $e^t \, C_{\phi_t}$ in $H^2(\CC_+)$, where
$$
\phi_t(s)=e^t s+ (e^t-1), \qquad  (s\in \CC_+).
$$
Namely,
\begin{equation}\label{similarity half-plane}
\mathcal{U} C_{\varphi_t} \mathcal{U}^{-1}= e^t C_{\phi_t}, \qquad (t\geq 0).
\end{equation}

\smallskip

\noindent Since we are interested in studying invariant subspaces, either for the entire semigroup or individual elements, we may
disregard  factors of the form $e^{\lambda t}$ for a fixed $\lambda \in \RR$.

\medskip

\noindent By means of the inverse Laplace transform we are led to consider the semigroup on $L^2(\RR_+)$
\begin{equation}\label{semigroup in half-line}
V_tg(x)= e^{-t}e^{-(1-e^{-t})x} g(e^{-t} x), \qquad (x>0, \ t \ge 0).
\end{equation}

\medskip

Now, proceeding as in \cite{GP09}, we may find a further equivalence with an operator on $L^2(\RR)$
using the unitary mapping $T: L^2(\RR) \to L^2(\RR_+)$ defined by
$$Th(x) = x^{-1/2}h(\log x), \qquad (x>0)$$
and
$$T^{-1}g(y)=e^{y/2} g(e^y), \qquad (y\in \RR).$$

\medskip

\noindent Accordingly,
\begin{equation}\label{eq:v2}
T^{-1}V_tT h(y)= e^{-t/2}e^{-(1-e^{-t})e^y}h(y-t),\qquad (y\in \RR),
\end{equation}
for $t\geq 0$.

Denoting by $\{S_t:\, t\geq 0\}$ the right-shift semigroup on $L^2(\RR)$:
$$
S_t f(y)=f(y-t), \qquad (y\in \RR),
$$
and recalling that if $w$ denotes a positive measurable function in $\RR$ the space $L^2(\RR, w(y)dy)$ consists of measurable functions in $\RR$ square-integrable respect to the measure $w(y)dy$, a key observation is the following

\begin{prop}\label{proposition right-shift}
The semigroup $\{\sigma_t:\, t\geq 0\}$ in $L^2(\RR)$ given by
\begin{equation}\label{sigma semigroup}
\sigma_t h(y)=e^{-(1-e^{-t})e^y}h(y-t), \qquad (y\in \RR),
\end{equation}
for $h\in L^2(\RR)$ is unitarily equivalent to the right-shift semigroup $\{S_t:\, t\geq 0\}$ acting on the weighted Lebesgue space $L^2(\RR, e^{-2(e^y-1)} \, dy)$.
\end{prop}

\begin{proof}
Let us denote the weight $w(y)=e^{-2(e^y-1)}$ for $y\in \RR$ and consider the unitary mapping $W:L^2(\RR) \to L^2(\RR, w(y) \, dy)$ given by
$$W h(y)=h (y)/\sqrt{w(y)}, \qquad (y\in \RR),$$
for $h\in  L^2(\RR)$. A computation shows that for any function $f\in L^2(\RR, w(y) \, dy)$ and $t>0$
\begin{align*}
W\sigma_t W^{-1} f(y) &= W \sigma_t \, f(y) e^{-(e^y-1)}\\
& = W e^{-(1-e^{-t})e^y} f(y-t) e^{-(e^{y-t}-1)}\\
& = e^{e^y-1}e^{-(1-e^{-t})e^y} f(y-t) e^{-(e^{y-t}-1)} \\
& = f(y-t),
\end{align*}
for $y\in \RR$. This yields the statement of the proposition.
\end{proof}

Figure \ref{fig:1} shows a plot of the weight function described in Proposition \ref{proposition right-shift}.

\begin{figure}
\includegraphics[height=8cm, width=11cm]{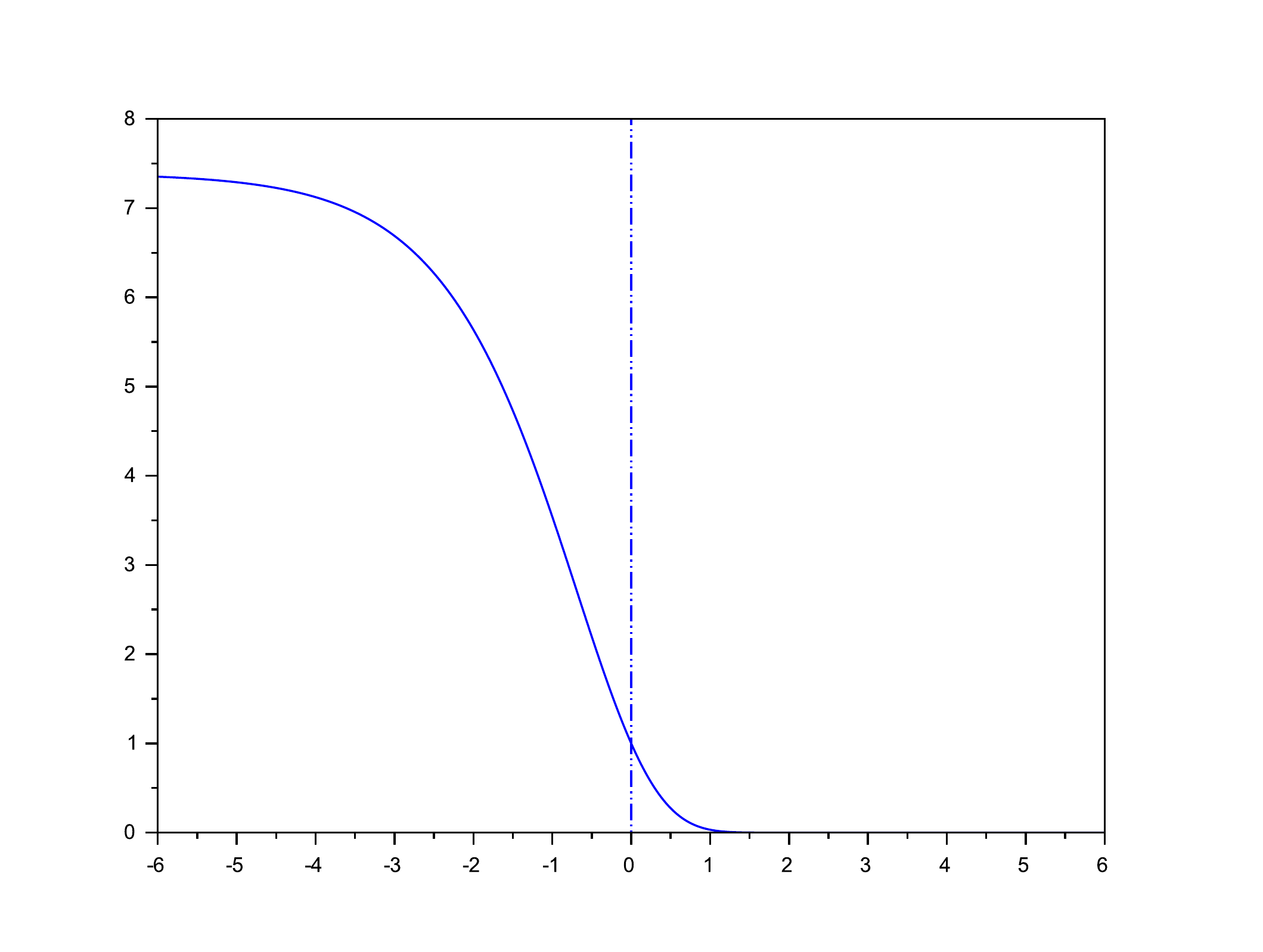}
\caption{The weight $e^{-2(e^y-1)}$ as a function of $y$.}
\label{fig:1}
\end{figure}

\medskip

As a by-product of equations \eqref{similarity half-plane}, \eqref{semigroup in half-line}, \eqref{eq:v2}, Proposition \ref{proposition right-shift} and Theorem \ref{thm semigroup}, if we denote by $\mathfrak{F}$ the unitary isomorphism $\mathfrak{F}=W T^{-1} \mathcal{L}^{-1} \mathcal{U}$ from $H^2(\DD)$ onto $L^2(\RR, e^{-2(e^y-1)} \, dy)$, the following result regarding the lattice of invariant subspaces of the Ces\`aro operator holds:

\begin{thm}\label{theorem semigroup R}
A closed subspace $M$ in $H^2(\DD)$ is invariant under the Ces\`aro operator if and only if $\mathfrak{F} M^\perp$ in $L^2(\RR, e^{-2(e^y-1)} \, dy)$ is invariant under the right-shift semigroup $\{S_t:\, t\geq 0\}$.
\end{thm}

\medskip

Accordingly, characterizing the lattice of invariant subspaces of the Ces\`aro operator in the Hardy space reduces to characterize the lattice of the right-shift semigroup in
$L^2(\RR, e^{-2(e^y-1)} \, dy)$.

\medskip

Though the lattice of the invariant subspaces of the right-shift semigroup acting on weighted Lebesgue spaces is only characterized for a very restricted subclass of weights (for instance the Beurling--Lax Theorem provides a characterization in $L^2(\RR_+)$, where the weight is the characteristic function $\chi_{(0,+\infty)}$ of $(0,+\infty)$), the question if such a lattice contains non-standard invariant subspaces has been extensively studied (see \cite{Do1, Do2, Do3, Do4}, \cite{GPR17, GR17}, \cite{Lax} or \cite{Ni}, for instance).

\medskip

Recall that given $a\in \RR\cup \{-\infty\} \cup \{\infty\}$, the  \textit{``standard invariant subspaces''} of $\{S_t:\, t\geq 0\}$ are given by
$$
L^2((a,\infty),w(y)\, dy)=\{f\in L^2(\RR, w(y)\, dy): f(y)=0 \mbox{ for a.~e. } y\leq a\}.
$$

\medskip

In \cite[Equation (8)]{Do1}, Domar proved that if the weight satisfies
$$
\displaystyle {\underline{\lim}}_{y\to -\infty} \frac{\log w(y)}{y}> -\infty
$$
then the lattice of invariant subspaces of $\{S_t:\, t\geq 0\}$ in $L^2(\RR, w(y)\,dy)$ contains non-standard invariant subspaces.

\medskip

\medskip

\noindent \emph{A word about notation}: Domar denotes by $L^2(\RR, w(y)\,dy)$ the space of measurable functions $f$ in $\RR$ such that
$f \,w \in L^2(\RR)$. Note that this does not affect the previous equation since it is enough to consider the positive function $w^{1/2}$.

\medskip

In our case, $w(y)= e^{-2(e^y-1)}$ for $y\in \RR$ and consequently, $\{S_t:\, t\geq 0\}$ has non-standard invariant subspaces in $L^2(\RR, e^{-2(e^y-1)} \, dy)$. Indeed, it is possible to exhibit many non-standard invariant subspaces in this case. In order to show them recall that, by means of the unitary equivalence $\LL: L^2(\mathbb{R}_+, \sqrt{2\pi}\, dt )\to H^2 (\CC_+)$, the Beurling--Lax Theorem asserts that a closed subspace $\mathcal{M}$ of $L^2(\RR_+)$ is invariant under every truncated right-shift to $L^2(\RR_+)$
$$
S_{\RR_+, \, \tau}f(t)=\left \{ \begin{array}{ll} 0 & \mbox { if  }0\leq t\leq \tau, \\
f(t-\tau)& \mbox { if  } t>\tau; \end{array} \right .
$$
$\tau \geq 0$, if and only if there exists an inner function $\Theta\in H^{\infty}(\CC_+)$ such that
$\LL \mathcal{M}= \Theta H^2(\CC_+)$
(see \cite[Cor. 6.5.5(2), p.149]{nik1}, for instance). Here, recall that an inner function $\Theta$ is an analytic function in $\CC_+$ with $|\Theta(z)|\leq 1$ for $z\in\CC_+$, such that the non-tangential limits exist and are of modulus 1 almost everywhere on the imaginary axis.

\begin{ex}
Let $T \in \RR$ fixed and write
$$L^2(\RR,w(y) \, dy)=L^2((-\infty,T),w(y) \, dy) \oplus L^2((T,\infty),w(y) \, dy),$$
the orthogonal direct sum of closed subspaces. Note that

\begin{enumerate}
\item $\exp(-2(e^T-1)) \le w(y) \le e^2$ on $(-\infty,T)$, so $L^2((-\infty,T),w(y) \, dy)$ is naturally isomorphic to
$L^2(-\infty,T)$.
\item If $M$ is a closed subspace of $L^2((-\infty,T),w(y) \, dy)$ invariant under all truncated right-shifts on $L^2((-\infty,T),w(y) \, dy)$, i.e.,
$$
S_{(-\infty,T), \, \tau}f(y)=\left \{ \begin{array}{ll} 0 & \mbox { if  } y-\tau>T, \\
f(y-\tau)& \mbox { if  } y-\tau\leq T, \end{array} \right .
$$

for $\tau \geq 0$, then $M \oplus  L^2((T,\infty),w(y) \, dy)$ is a closed subspace of $ L^2(\RR,w(y) \, dy)$
invariant under all right shifts.
\end{enumerate}

Now, the Beurling--Lax theorem provides a large class of nonstandard invariant subspaces $M$: take the ``twisted'' Laplace transform
\beq\label{eq:twist}
\tilde\LL f(s) = \int_{-T}^\infty e^{-su} f(-u) \, du
\eeq
which gives an isomorphism from $L^2((-\infty,T),w(y) \, dy) $ onto $e^{sT}H^2(\CC_+)$. Then any subspace of the form $\widetilde \LL^{-1} e^{sT} K_\Theta$ is invariant under all truncated right shifts $S_{(-\infty,T), \, \tau}$ where $\Theta \in H^\infty(\CC_+)$ is inner and $K_\Theta = H^2(\CC_+) \ominus \Theta H^2(\CC_+)$ is the associated model space (these calculations are easiest to follow when $T=0$, and the general case is a shifted version.)

\medskip

As an explicit example, let $K_\Theta$ be spanned by the reproducing kernel $s \mapsto 1/(s+\overline \lambda)$
for  $\lambda \in \CC_+$, so that $M$ is the one-dimensional space of $L^2((-\infty,T),w(y)\, dy)$ spanned by $e^{\overline \lambda t}$; then $M \oplus  L^2((T,\infty),w(y) \, dy)$ is a non-standard invariant subspace for all right shifts.
\end{ex}



While a theorem of Aleman and Koreblum \cite{AK08} asserts that the analytic Volterra operator is unicellular in $H^p$-spaces, as consequence of the previous considerations we have a new deduction of the following known result:

\begin{cor}\label{no-nunicellular}
The Ces\`aro operator $\mathcal{C}$ is not a unicellular operator in $H^2(\DD)$.
\end{cor}

\smallskip

In this regard, the feature that the Ces\`aro operator $\mathcal{C}$ is not a unicellular operator on $H^2(\DD)$ can be also deduced from \cite{BHS}, as the referee kindly pointed out to us. Indeed, it follows from the result that the point spectrum of $I-\mathcal{C}^*$ in $H^2(\DD)$ is $\DD$ along with the fact that any operator on a Hilbert space which has at least two eigenvalues cannot be unicellular. Likewise, in \cite[Corollary 6]{KT74}, the authors constructed two nonzero invariant subspaces of $\mathcal{C}$ whose intersection is zero space.

On the other hand, it is worth pointing out that the classifying the invariant subspaces turns out to be completely different if one considers other semigroups studied in the context of Ces\`aro-like operators, as  in the following remark:

\begin{rem}
In \cite{AS}, the authors considered the composition operator group on $H^2(\CC_+)$ corresponding to the flow on $\CC_+$ given by
$$\phi_t(s)=e^{-t}s,\qquad (s\in \CC_+)$$
$t \in \RR$, in a broader context of studying Ces\`aro-like operators.

Proceeding similarly as before, the transformed semigroup on $L^2(0,\infty)$ is given by
\[
\tilde{V}_{t}g(x)=e^{t}g(e^t x), \qquad (x >0, \ t \in \RR  \mbox{ and } g\in L^2(0,\infty)),
\]
which transferred to $L^2(\RR)$ is
\[
T^{-1}\tilde{V}_t T h(y) = e^{t/2}h(y+t), \qquad (y\in \RR, \ t \in \RR),
\]
for $h\in L^2(\RR)$.

The subspaces $M$ invariant under the group $(\tau_t)_{t \in \RR}=(T^{-1}\tilde{V}_tT )_{t \in \RR}$ were essentially classified by Lax
-- the factors $e^{t/2}$ are irrelevant -- and can be found, with a slightly different notation in \cite[Cor~6.5.4, p.~149]{nik1}.
There are two types:
\begin{enumerate}
\item $1$-invariant subspaces, i.e., $\tau_t M \subset M$ for all $t<0$ but not for all $t \in \RR$.
These have the form $M=\FF q H^2(\Pi^+)$, where $q$ is measurable with $|q|=1$ almost everywhere
and here $\Pi^+$ denotes the upper half-plane;
\item $2$-invariant subspaces, i.e., $\tau_t M \subset M$ for all $t \in \RR$. These have the form
$M=\FF \chi_E L^2(\RR)$ for some measurable subset $E \subset \RR$.
\end{enumerate}
Here $\FF$ denotes the Fourier transform but, alternatively, one can use the bilateral Laplace transform and express the subspaces in terms of $L^2(i\RR)$
and the space $H^2(\CC_+)$ of the right half-plane.

\medskip

\noindent Likewise, in this case the invariant subspaces of the form $\tilde \LL^{-1}K_\Theta \oplus L^2(\RR_+)$ are 1-invariant subspaces as described above,
corresponding in $L^2(i\RR)$ to $\overline\Theta K_\Theta \oplus H^2(\CC_+)= \overline\Theta H^2(\CC_+)$.

\end{rem}

Finally, as an application of Theorem \ref{theorem semigroup R}, we present a characterization of the finite codimensional invariant subspaces of the Ces\`aro operator $\mathcal{C}$ in $H^2(\DD)$. Of particular relevance will be a theorem of Domar \cite{Do2} which states that the lattice of the invariant subspaces of $\{S_{\tau}:\tau\geq 0\}$ consists of just the standard invariant subspaces in $L^2(\RR_+, w(x)\, dx)$
whenever:
\begin{enumerate}
\item $w$ is a positive continuous function in $\RR_+$ such that $\log w$ is concave in $[c,\infty)$ for some $c>0$.
\item $\displaystyle \lim_ {x\to\infty}\frac{-\log w(x)}{x}=\infty.$
\item $\displaystyle \lim_ {x\to\infty}\frac{\log|\log w(x)|-\log x}{\sqrt{\log x}}=\infty.$
\end{enumerate}

\medskip

\begin{thm}
A finite codimensional closed subspace $M$ in $H^2(\DD)$ is invariant under the Ces\`aro operator if and only if $\mathfrak{F} M^\perp$ in $L^2(\RR, e^{-2(e^y-1)} \, dy)$ is spanned by a finite subset of functions given by
$$
\bigcup_{\lambda \in \Lambda} \{y^k e^{\lambda y}:\; k = 0,1,2,\ldots, n_\lambda \}$$
where $\Lambda\subset \CC_+$ is a finite set and $n_{\lambda} \geq 0$ for each $\lambda \in \Lambda$.
\end{thm}

Before proceeding with the proof, let us introduce the notation
$$f_{\lambda,k}(y) = y^k e^{\lambda y}$$
for  $y\in \RR, \lambda \in \CC_+$
and $k=0, 1, 2, \dots$ Note that
$$(S_t f_{\lambda,k})(y) = f_{\lambda,k}(y-t)=\sum_{j=0}^k \binom{k}{j} (-t)^{k-j} e^{-\lambda t} y^j e^{\lambda y},$$
so
\begin{equation}\label{image f}
S_t f_{\lambda,k}= \sum_{j=0}^k \binom{k}{j} (-t)^{k-j}  e^{-\lambda t} f_{\lambda,j}
\end{equation}
for  any  $\lambda \in \CC_+$ and $k=0, 1, 2, \dots$

\begin{proof}
Suppose first that $M$ is a finite codimensional closed subspace in $H^2(\DD)$ invariant under $\mathcal{C}$. Theorem \ref{theorem semigroup R} yields that $N=\mathfrak{F} M^\perp$ is a finite dimensional subspace of $L^2(\RR, e^{-2(e^y-1)} \, dy)$ invariant under all right shifts. Thus, $P_- N$, the projection onto $L^2((-\infty,0), e^{-2(e^y-1)} \, dy)\cong L^2(-\infty,0)$, is a  finite dimensional subspace invariant under all the  truncated  right shifts.

Thus, by the Beurling--Lax Theorem, $P_- N$ corresponds to a model space and in particular is spanned by a finite set of functions of the form
$$y^k e^{\lambda y}, \mbox{ for } y\in (-\infty, 0)$$
where $k = 0,1,2,\ldots, n_{\lambda}$ for $\lambda\in \Lambda \subset \CC_+$. We now show that $N$ is spanned by what we shall call the ``natural extension'' to $\RR$ of such functions as elements of $L^2(\RR, e^{-2(e^y-1)}\, dy)$, namely,  $y^k e^{\lambda y} \mbox{ for } y\in \RR$.

Observe that since N is ﬁnite dimensional, by the Domar theorem aforementioned,
$$N\cap L^2((0,\infty), e^{-2(e^y-1)}  \, dy)=\{0\}.$$
Therefore, there exist $h_{\lambda, k} \in N$ such that $P_- f_{\lambda, k}= h_{\lambda, k}$ and $N$ is spanned by $h_{\lambda, k}$. Let $N_1$ be spanned by $f_{\lambda, k}$ with the same $k$ and $\lambda$. By \eqref{image f}, $N_1$ is invariant under all right shifts $S_t$. Then $\Span\{N, N_1\}$ is a finite dimensional invariant subspace of all right shifts $S_t$. Upon applying Domar's theorem again,
$$
\Span\{N, N_1\} \cap  L^2((0, \infty), e^{-2(e^y-1)}\, dy)=\{0\}.
$$
Since
$$h_{\lambda, k}- f_{\lambda, k} \in  \Span\{N, N_1\} \cap  L^2((0, \infty), e^{-2(e^y-1)}\, dy),$$
we conclude that $h_{\lambda, k}=f_{\lambda, k}$. Thus, $N$ is spanned by $f_{\lambda, k}$.




\medskip

For the converse, assume $N=\mathfrak{F} M^\perp$ in $L^2(\RR, e^{-2(e^y-1)} \, dy)$ is a finite dimensional subspace spanned  by a finite subset
of the form $\bigcup_{\lambda \in \Lambda} \{y^k e^{\lambda y}:\; k = 0,1,2,\ldots, n_\lambda \}$
where $\lambda \in \Lambda \subset \CC_+$ is finite. Equation \eqref{image f} yields that $N$ is invariant under all right shifts $S_t$. Now,
Theorem \ref{theorem semigroup R} and the fact that $\mathfrak{F}$ is an isomorphism yield that $M$ is a finite codimensional closed subspace invariant under $\mathcal{C}$, which completes the proof.
\end{proof}

\begin{rem}
Note that in the case that $N=\mathfrak{F} M^\perp$ is infinite-dimensional, the arguments above giving the structure of
$N$ in terms of the structure of $P_-N$ fail because we can no longer assume that $P_-N$ is closed. However, it is of interest to note that
the {\em closure} $\overline{P_- N}$ in $L^2(-\infty,0)$ has
the same property of invariance under all truncated right shifts, and corresponds to a model space.

For instance, if $B$ is a Blaschke product in $\CC_+$ with the set of zeros $\Lambda$ and multiplicities $n_{\lambda} + 1$ for $\lambda \in \Lambda$ with $n_{\lambda}\in \{0, 1, 2, \dots\}$ and we set
$$N_B =\overline{\Span\{f_{\lambda, k}:\, \lambda \in \Lambda, k = 0, 1, \dots, n_{\lambda}\}}^{L^2(\RR, e^{-2(e^y-1)} \, dy)},$$
clearly $N_B$ is invariant under all right shifts. Nevertheless,
$$\overline{\tilde\LL P_- N_B}= \overline{\Span\{\widetilde{\mathcal{L}}P_- f_{\lambda, k}: \lambda \in \Lambda, k=0, 1,\dots, n_{\lambda} \}}^{H^2(\CC_+)}=K_B,$$
where $\tilde\LL$ is defined in \eqref{eq:twist} with $T=0$. Consequently, $N_B\neq  L^2(\RR, e^{-2(e^y-1)} \, dy)$, and if $B_1\neq B_2$ are two Blaschke products, then $N_{B_1}\not = N_{B_2}$.
\end{rem}

\subsection{A final remark regarding the lattice of the invariant subspaces of $\mathcal{C}$}

As we have just noted, the approach addressed in the previous theorem fails if $P_- N$ is not closed. Indeed, the following example shows that $P_- N$ need not be closed even if $N$ is a closed shift-invariant subspace of $L^2(\RR,e^{-2(e^y-1)} \, dy)$ showing, somehow, the limits of such an approach.

\smallskip

Let $\lambda>0$ and denote by $e_\lambda$ the function $e_\lambda:y\in \RR \mapsto e^{\lambda y}$.
Since $1 < e^{-2(e^y-1)} < e^2$ for $y<0$, we have
%
\beq\label{eq:el1}
\|e_\lambda\|^2_{L^2((-\infty,0),e^{-2(e^y-1)} \, dy)}   \approx
\|e_\lambda\|^2_{L^2(-\infty,0)} \approx
1/\lambda.
\eeq
On the other hand,
\begin{eqnarray}
\|e_\lambda\|^2_{L^2(\RR,e^{-2(e^y-1)} \, dy)}
 & \ge &
\int_1^2 e^{2\lambda y} e^{-2(e^y-1)} \, dy
\ge
e^{2\lambda}e^{-2(e^2-1)}.
\label{eq:elambda}
\end{eqnarray}

Now take $N$ to be the closed linear span in $L^2(\RR,e^{-2(e^y-1)} \, dy)$ of $\{e_{n^2}: n \in \NN\}$, that is,
$$
N=\overline{\Span \{e_{n^2}: n \in \NN \}}^{L^2(\RR,e^{-2(e^y-1)} \, dy)}.
$$
Now the twisted Laplace transform given in \eqref{eq:twist}, with $T=0$, provides an isomorphism
from $ L^2((-\infty,0),e^{-2(e^y-1)} \, dy)$ onto $H^2(\CC_+)$ transforming  $e_\lambda$ to $1/(s+\lambda)$.
By an argument similar to that used in proving the classical M\" untz--Sz\'asz theorem
it follows that $\overline{P_-N}\subsetneq L^2((-\infty,0),e^{-2(e^y-1)} \, dy)$ since there are functions orthogonal
to each $1/(s+n^2)$, for example, $1/(s+2)$ times a Blaschke product with zeros at $\{n^2: n \in \NN\}$. \\

Now $(\ker P_- )\cap N$ is a closed shift-invariant
subspace of $L^2((0,\infty), e^{-2(e^y-1)} \, dy)$ and hence, by Domar's theorem, a standard subspace. It must be $\{0\}$ (again this follows
from a M\H untz--Sz\'asz argument)
so  the restriction $P_- : N \to L^2((-\infty,0),e^{-2(e^y-1)} \, dy)$ is injective.
Finally, by Banach's open mapping theorem the norm estimates in \eqref{eq:el1} and \eqref{eq:elambda} show that it cannot have a closed range.

\medskip

The following proposition characterizes when $P_-N$ is a closed subspace of $L^2((-\infty,0), e^{-2(e^y-1)} \, dy)$.

\medskip

\begin{prop} With the above notation,  $P_- N$ is closed if and only if the (not necessarily direct) sum of the
two closed subspaces $N$ and $L^2((0,\infty), e^{-2(e^y-1)} \, dy)$ is closed.
\end{prop}

\begin{proof}
It is well known that if $P:H \to K$ is a projection onto a subspace $K$ of a Hilbert space $H$ then $K$ is automatically
closed: since if $(x_n)$ is a sequence in $K$ tending to $x \in H$, we have $x_n=Px_n \to Px$, so $x=Px \in K$.

Now if $H:=N+L^2((0,\infty), e^{-2(e^y-1)} \, dy)$ is closed, the projection $P_-$ maps $H$ to itself and its image is $P_-N$.
Conversely, if $P_-N$ is closed then $N + L^2((0,\infty), e^{-2(e^y-1)} \, dy)=P_-N \oplus L^2((0,\infty), e^{-2(e^y-1)} \, dy)$, the orthogonal direct sum of
two closed subspaces, and is therefore closed.
\end{proof}

\section{Functions of the Ces\`aro operator}\label{section 3}

  Suppose that $(T(t))_{t \ge 0}$ is a $C_0$ semigroup with $\|T(t)\| \le e^{mt}$ for some $m<1$
  and infinitesimal generator $A$.

It is a standard fact that
  \[
  \int_0^\infty e^{-\lambda t} T(t) \, dt = (\lambda-A)^{-1}
  \]
  provided that $\re \lambda > m$.

  Now consider the operator
  \[
  B= \frac{1}{\sqrt\pi}\int_0^\infty \frac{e^{-\lambda t}}{\sqrt t} T(t) \, dt.
  \]
  Then
  \begin{align*}
  B^2  &= \frac{1}{\pi}\int_0^\infty \frac{e^{-\lambda t}}{\sqrt t} T(t) \, dt \int_0^\infty \frac{e^{-\lambda u}}{\sqrt u} T(u) \, du \\
  &= \frac{1}{\pi} \int_{w=0}^\infty e^{-\lambda w} T(w) \, dw\int_{t=0}^w \frac{1}{\sqrt t \sqrt{w-t}} \, dt
  \end{align*}
  with $w=t+u$. The second integral is $\pi$ (use the substitution $t=w \sin^2 \theta$)
  and so $B^2= (\lambda-A)^{-1}$. Similarly for $(-B)^2$, of course.

  This gives an alternative way of looking at a result in \cite{MPR} on the square roots of the Ces\`aro operator, by using the composition semigroup and the observations
  in Section \ref{section 2}. Not surprisingly, it is linked with the fact
  that the Laplace transform of $e^{at}$ is $1/(s-a)$ and the Laplace transform of $e^{at}/\sqrt t$ is
 $ \sqrt{\pi}/\sqrt{s-a}$. There is a more general functional calculus
 available here, but this calculation at least can be done directly.\\

 Indeed the Phillips functional calculus \cite[Rem. 3.3.3]{haase} allows us,
 given a bounded semigroup $(T(t))_{t \ge 0}$ of operators on a Banach space $X$,
  to associate
 an operator $f(A)$ to a function $f$ that is the Laplace transform of a Borel measure $\mu$
 on $[0,\infty)$
 of bounded variation, by the formula
 \[
 f(A)x = \int_{[0,\infty)} T(t)x \, d\mu(t) \qquad (x \in X).
 \]
Note that the convention in \cite{haase} is that the generator is $-A$, rather than $A$, and we have allowed
for that in the discussion below.
In particular,
  we have
\beq\label{eq:integral}
\int_0^\infty \frac{e^{-\lambda t}}{t^{1-\beta}} \, dt = \lambda^{-\beta}\Gamma(\beta)
\eeq
for $\re \lambda,\re \beta>0$, so that
  \[
  \int_0^\infty \frac{e^{-t}e^{at}}{t^{1-\beta}} \, dt = (1-a)^{-\beta}\Gamma(\beta)
  \]
  for $\re a<1$ and $\re \beta>0$, from which we obtain
  \beq\label{eq:keyfc}
  \int_0^\infty \frac{e^{-t}T(t)}{t^{1-\beta}} \, dt = (I-A)^{-\beta}\Gamma(\beta)
  \eeq
  for $\re \beta >0$. A similar formula holds on replacing $T(t)$ by $T(t)^*$ and $A$ by $A^*$.

 Recalling that $\mathcal{C}^*=(I-A)^{-1}$, we have the following:

  \begin{thm}
  For $\re \beta>0$
 let $M_\beta$ denote the matrix of $\mathcal{C}^\beta$ (as defined using \eqref{eq:keyfc})
  with respect to the standard orthonormal basis $(z^j)_{j=0}^\infty$.
Then
\beq\label{eq:matrix}
(M_\beta)_{i,j}=
\begin{cases}
0 &  \text{if } i<j ,\\
\displaystyle {i \choose j} \sum_{k=0}^{i-j}   (-1)^k {i-j \choose k} (j+k+1)^{-\beta} &  \text{if } i \ge j.
\end{cases}
\eeq
\end{thm}

\beginpf
The $(j,i)$ entry of the matrix for $(I-A)^{-\beta}$
 is the coefficient of $z^j$ in
\[
\frac{1}{\Gamma(\beta)} \int_0^\infty \frac{e^{-t}}{t^{1-\beta}}(e^{-t}z+(1-e^{-t}))^i \, dt,
\]
namely, $0$ for $j >i$ and otherwise
\[
\frac{1}{\Gamma(\beta)}  \int_0^\infty \frac{e^{-t}}{t^{1-\beta}} {i \choose j} e^{-jt}(1-e^{-t})^{i-j} \, dt
= \frac{1}{\Gamma(\beta)} {i \choose j} \int_0^\infty \frac{e^{(-1-j)t}}{t^{1-\beta}}\sum_{k=0}^{i-j} {i-j \choose k} (-1)^k e^{-kt} \, dt.
\]
 Now, using \eqref{eq:integral}    we obtain \eqref{eq:matrix}.
\endpf
In the case $\beta=1/2$ this agrees with the formula in \cite{MPR}.

\begin{rem}
It is clear from the functional
calculus that every subspace for   $\mathcal{C}$ is also an invariant subspace for $\mathcal{C}^{\beta}$ for $\re \beta>0$.
Since invariant subspaces for $\mathcal{C}^{1/n}$  are clearly invariant subspaces for $\mathcal{C}$ for $n=1,2,\ldots$,
we may conclude that  $\mathcal{C}$ and $\mathcal{C}^{1/n}$ have the same lattice of invariant subspaces.
\end{rem}

\section*{Acknowledgements}
The authors are grateful to a referee for carefully reading the manuscript and providing some
extremely helpful comments which improved its readability.

\end{document}